\documentclass[twoside,reqno]{amsart}

\usepackage{latexsym}

\newcommand{\qdn}{\hspace*{-1.5mm}}
\newcommand{\qqdn}{\hspace*{-2.5mm}}
\newcommand{\xqdn}{\hspace*{-5.0mm}}
\newcommand{\xxqdn}{\hspace*{-10mm}}

\newcommand{\binm}{\binom}

\newcommand{\nnm}{\nonumber}
\newcommand{\be}{\begin{equation}}
\newcommand{\ee}{\end{equation}}
\newcommand{\ba}{\begin{array}}
\newcommand{\ea}{\end{array}}
\newcommand{\bmn}{\begin{eqnarray}}
\newcommand{\emn}{\end{eqnarray}}
\newcommand{\bnm}{\begin{eqnarray*}}
\newcommand{\enm}{\end{eqnarray*}}
\newcommand{\bln}{\begin{subequations}}
\newcommand{\eln}{\end{subequations}}

\newtheorem{thm}{Theorem}

\newtheorem{entry}{Entry}

\newcommand{\bbtm}[4]{\bibitem{kn:#1}{#2,}~{#3,}~{#4.}}
\newcommand{\cito}[1]{\cite{kn:#1}}
\newcommand{\citu}[2]{\cite[#2]{kn:#1}}

\begin{document} 
{
\title{Summation formulas involving harmonic numbers with even or odd indexes}
\author{$^{*}$Chuanan Wei$^{a}$, Dianxuan Gong$^{b}$, Lily Li Liu$^{c}$}


\dedicatory{$^{A}$Department of Medical Informatics,
 Hainan Medical University, Haikou, China\\
 $^{B}$College of Sciences,
  North China University of Science and Technology, Tangshan,
  China\\
  $^{C}$School of Mathematical Sciences,
 Qufu Normal University, Qufu, China}
\thanks{The corresponding author$^{*}$ \emph{Email address}: weichuanan78@163.com (Chuanan Wei),
gdx216@163. com (Dianxuan Gong), liulily@mail.qfnu.edu.cn (Lily Li
Liu)}


\begin{abstract}
By means of the derivative operator and Chu-Vandermonde convolution,
four families of summation formulas involving harmonic numbers with
even or odd indexes are established.
\end{abstract}

\maketitle\thispagestyle{empty}
\markboth{C. Wei, D. Gong, Lily L. Liu}
         {Harmonic numbers}

\section{Introduction}

For a nonnegative integer $n$, define harmonic numbers to be
\[H_{0}=0\quad\text{and}\quad
H_{n}=\sum_{k=1}^n\frac{1}{k} \quad\text{when}\quad
n\in\mathbb{N},\] which arise from truncation of the harmonic series
\[\sum_{k=1}^{\infty}\frac{1}{k}.\]
 Harmonic numbers are important in many branches of number theory and appear in the expressions
of various special functions.

 For a differentiable function $f(x)$, define the derivative operator
$\mathcal{D}_x$ by
 \bnm
&&\mathcal{D}_xf(x)=\frac{d}{dx}f(x)\Big|_{x=0}.
 \enm
 Then it is not difficult to show that
$$\mathcal{D}_x\:\binm{x+r}{s}=\binm{r}{s}\big\{H_r-H_{r-s}\big\},$$
where $r,s\in\mathbb{N}_0$ with $s\leq r$.

As pointed out by Richard Askey (cf. \cito{andrews}), expressing
harmonic numbers in accordance with differentiation of binomial
coefficients can be traced back to Issac Newton.
 In 2003, Paule and Schneider
\cito{paule} computed the family of series:
 \bnm
W_n(\alpha)=\sum_{k=0}^n\binm{n}{k}^{\alpha}\{1+\alpha(n-2k)H_k\}
 \enm
with $\alpha=1,2,3,4,5$ by combining this way with Zeilberger's
algorithm for definite hypergeometric sums. According to the
derivative operator and the hypergeometric form of Andrews'
$q$-series transformation, Krattenthaler and Rivoal
\cito{krattenthaler} deduced general Paule-Schneider type identities
with $\alpha$ being a positive integer.  More results from
differentiation of binomial coefficients can be seen in the papers
\cite{kn:sofo-a,kn:wang-b,kn:wei}. For different ways and related
harmonic number identities, the reader may refer to
\cite{kn:ahlgren,kn:chen,kn:choi,kn:liu,kn:schneider,kn:sofo-b,kn:wang-a,kn:xu}.
It should be mentioned that Sun \cite{kn:sun-a,kn:sun-b} showed
recently some congruence relations concerning harmonic numbers to
us.

There are numerous binomial identities in the literature. Thereinto,
Chu-Vandermonde convolution (cf. \citu{andrews-r}{p. 67}) can be
stated as
 \bmn\label{vandermonde}
 \quad\sum_{k=0}^n\binm{x}{k}\binm{y}{n-k}=\binm{x+y}{n}.
 \emn

Inspired by the works just mentioned, we shall establish, in terms
of the derivative operator and \eqref{vandermonde}, closed
expressions for the following four families of series with even or
odd indexes:
 \bnm
 &&\sum_{k=0}^{n}(-4)^k\frac{\binm{n}{k}}{\binm{2k}{k}}k^tH_{2k},\\
 &&\sum_{k=0}^{n}(-4)^k\frac{\binm{n}{k}}{\binm{1+2k}{k}}k^tH_{1+2k},\\
&&\sum_{k=0}^{n}\bigg(-\frac{1}{4}\bigg)^k\binm{n}{k}\binm{2k}{k}k^tH_{2k},\\
 &&\sum_{k=0}^{n}\bigg(-\frac{1}{4}\bigg)^k\binm{n}{k}\binm{1+2k}{k}k^tH_{1+2k},
 \enm
where $t\in\mathbb{N}_0$. In order to avoid appearance of
complicated expressions, our explicit formulas are offered only for
$t=0,1,2$.
\section{The first family of summation formulas involving harmonic numbers}
\begin{thm}\label{thm-a}
\bnm
 \qqdn\sum_{k=0}^{n}(-4)^k\frac{\binm{n}{k}}{\binm{2k}{k}}H_{2k}
 =\frac{2H_{2n}-H_n}{2(2n-1)}-\frac{4n}{(2n-1)^2}.
 \enm
\end{thm}

\begin{proof}
Perform the replacements $x\to-a-1$ and $y\to b+n$ in
\eqref{vandermonde} to obtain
 \bmn\label{chu-vandermonde}
 &&\xxqdn\qqdn\qdn\sum_{k=0}^n(-1)^k\binm{a+k}{k}\binm{b+n}{n-k}=\binm{b-a-1+n}{n}.
 \emn
 The case $a=\frac{x}{2}$, $b=\frac{-x-1}{2}$ of
 it reads as
 \bnm
 \quad\sum_{k=0}^n(-1)^k\binm{\frac{x}{2}+k}{k}\binm{\frac{-x-1}{2}+n}{n-k}=\binm{-x-\frac{3}{2}+n}{n}.
 \enm
 Applying the derivative operator $\mathcal{D}_x$ to both sides of
the last equation, we get
 \bmn\label{harmornic-a}
 &&\sum_{k=0}^{n}(-1)^k\binm{-\frac{1}{2}+n}{n-k}
  \bigg\{\frac{1}{2}H_k+\frac{1}{2}\sum_{i=1}^k\frac{1}{i-\frac{1}{2}}-\frac{1}{2}\sum_{i=1}^n\frac{1}{i-\frac{1}{2}}\bigg\}
   \nnm\\
&&=-\binm{-\frac{3}{2}+n}{n}\sum_{i=1}^n\frac{1}{i-\frac{3}{2}}.
 \emn
By means of the two relations
 \bmn\label{relation-a}
\binm{-\frac{1}{2}+n}{n-k}=4^{k-n}\binm{n}{k}\frac{\binm{2n}{n}}{\binm{2k}{k}},
 \\\label{relation-b}
\frac{1}{2}H_k+\frac{1}{2}\sum_{i=1}^k\frac{1}{i-\frac{1}{2}}=H_{2k},
 \emn
\eqref{harmornic-a} can be manipulated as
\bnm
 \quad\sum_{k=0}^{n}(-4)^k\frac{\binm{n}{k}}{\binm{2k}{k}}H_{2k}
 =\frac{2H_{2n}-H_n}{2n-1}-\frac{4n}{(2n-1)^2}+\frac{2H_{2n}-H_n}{2}\sum_{k=0}^{n}(-4)^k\frac{\binm{n}{k}}{\binm{2k}{k}}.
 \enm
 Calculating the series on the right hand side by \eqref{chu-vandermonde}, we gain Theorem \ref{thm-a}.
\end{proof}

\begin{thm}\label{thm-b}
 \bnm
 \qquad\sum_{k=0}^{n}(-4)^k\frac{\binm{n}{k}}{\binm{2k}{k}}kH_{2k}
 =\frac{n\{H_{n}-2H_{2n}\}}{(2n-1)(2n-3)}+\frac{n(20n^2-24n-1)}{(2n-1)^2(2n-3)^2}.
 \enm
\end{thm}

\begin{proof}
 It is ordinary to find that
 \bnm
 &&\xxqdn\sum_{k=0}^n(-1)^kk\binm{a+k}{k}\binm{b+n}{n-k}\\
 &&\xxqdn\:=\:(a+1)\sum_{k=1}^n(-1)^k\binm{a+k}{k-1}\binm{b+n}{n-k}\\
 &&\xxqdn\:=\:-(a+1)\sum_{k=0}^{n-1}(-1)^k\binm{a+1+k}{k}\binm{b+n}{n-1-k}.
 \enm
Evaluate the series on the right hand side by
\eqref{chu-vandermonde} to achieve
 \bmn\label{chu-vandermonde-a}
\qquad\sum_{k=0}^n(-1)^kk\binm{a+k}{k}\binm{b+n}{n-k}=\frac{(a+1)n}{a+1-b-n}\binm{b-a-1+n}{n}.
 \emn
 The case $a=\frac{x}{2}$, $b=\frac{-x-1}{2}$ of
 it is
 \bnm
 \qquad\quad\sum_{k=0}^n(-1)^kk\binm{\frac{x}{2}+k}{k}\binm{\frac{-x-1}{2}+n}{n-k}=\frac{(x+2)n}{2x+3-2n}\binm{-x-\frac{3}{2}+n}{n}.
 \enm
 Applying the derivative operator $\mathcal{D}_x$ to both sides of
the last equation, we have
 \bnm
 &&\xxqdn\qdn\sum_{k=0}^{n}(-1)^kk\binm{-\frac{1}{2}+n}{n-k}
  \bigg\{\frac{1}{2}H_k+\frac{1}{2}\sum_{i=1}^k\frac{1}{i-\frac{1}{2}}-\frac{1}{2}\sum_{i=1}^n\frac{1}{i-\frac{1}{2}}\bigg\}
   \nnm\\
&&\xxqdn\qdn\:=\:\frac{2n}{2n-3}\binm{-\frac{3}{2}+n}{n}\sum_{i=1}^n\frac{1}{i-\frac{3}{2}}-\frac{n(2n+1)}{(2n-3)^2}\binm{-\frac{3}{2}+n}{n}.
 \enm
In accordance with \eqref{relation-a} and \eqref{relation-b}, the
last equation can be restated as
\bnm
 \:\sum_{k=0}^{n}(-4)^k\frac{\binm{n}{k}}{\binm{2k}{k}}kH_{2k}
 &&\xqdn\!=\frac{n(20n^2-24n-1)}{(2n-1)^2(2n-3)^2}-\frac{2n\{2H_{2n}-H_n\}}{(2n-1)(2n-3)}\\
 &&\xqdn\!+\:
 \frac{2H_{2n}-H_n}{2}\sum_{k=0}^{n}(-4)^k\frac{\binm{n}{k}}{\binm{2k}{k}}k.
 \enm
 Computing the series on the right hand side by \eqref{chu-vandermonde-a}, we attain Theorem \ref{thm-b}.
\end{proof}

\begin{thm}\label{thm-c}
 \bnm
 \sum_{k=0}^{n}(-4)^k\frac{\binm{n}{k}}{\binm{2k}{k}}k^2H_{2k}
 &&\xqdn=\:\frac{n(2n+1)\{2H_{2n}-H_n\}}{(2n-1)(2n-3)(2n-5)}\\
 &&\xqdn-\:\:\frac{n(96n^4-280n^3+60n^2+182n-13)}{(2n-1)^2(2n-3)^2(2n-5)^2}.
 \enm
\end{thm}

\begin{proof}
 It is routine to verify that
 \bnm
 &&\xxqdn\sum_{k=0}^n(-1)^kk^2\binm{a+k}{k}\binm{b+n}{n-k}\\
 &&\xxqdn\:=\:(a+1)\sum_{k=1}^n(-1)^kk\binm{a+k}{k-1}\binm{b+n}{n-k}\\
 &&\xxqdn\:=\:-(a+1)\sum_{k=0}^{n-1}(-1)^k(k+1)\binm{a+1+k}{k}\binm{b+n}{n-1-k}\\
&&\xxqdn\:=\:-(a+1)\sum_{k=0}^{n-1}(-1)^k\binm{a+1+k}{k}\binm{b+n}{n-1-k}\\
&&\xqdn-\:(a+1)\sum_{k=0}^{n-1}(-1)^kk\binm{a+1+k}{k}\binm{b+n}{n-1-k}.\\
 \enm
Calculate respectively the two series on the right hand side by
\eqref{chu-vandermonde} and \eqref{chu-vandermonde-a} to obtain
 \bmn\label{chu-vandermonde-b}
\quad\sum_{k=0}^n(-1)^kk^2\binm{a+k}{k}\binm{b+n}{n-k}=\frac{n(a+1)(an+n-b)}{(b-a-2+n)(b-a-1+n)}\binm{b-a-1+n}{n}.
 \emn
The case $a=\frac{x}{2}$, $b=\frac{-x-1}{2}$ of
 it can be written as
\bnm
 \quad\sum_{k=0}^n(-1)^kk^2\binm{\frac{x}{2}+k}{k}\binm{\frac{-x-1}{2}+n}{n-k}=\frac{n(x+2)(1+2n+x+nx)}{(2x+3-2n)(2x+5-2n)}\binm{-x-\frac{3}{2}+n}{n}.
 \enm
 Applying the derivative operator $\mathcal{D}_x$ to both sides of
the last equation, we get
 \bnm
 &&\xqdn\sum_{k=0}^{n}(-1)^kk^2\binm{-\frac{1}{2}+n}{n-k}
  \bigg\{\frac{1}{2}H_k+\frac{1}{2}\sum_{i=1}^k\frac{1}{i-\frac{1}{2}}-\frac{1}{2}\sum_{i=1}^n\frac{1}{i-\frac{1}{2}}\bigg\}
   \nnm\\
&&\xqdn\:\:=\frac{n(16n^3-20n^2-36n+13)}{(2n-3)^2(2n-5)^2}\binm{-\frac{3}{2}+n}{n}-
\frac{2n(2n+1)}{(2n-3)(2n-5)}\binm{-\frac{3}{2}+n}{n}\sum_{i=1}^n\frac{1}{i-\frac{3}{2}}.
 \enm
According to \eqref{relation-a} and \eqref{relation-b}, the last
equation can be reformulated as
 \bnm
 &&\xqdn\sum_{k=0}^{n}(-4)^k\frac{\binm{n}{k}}{\binm{2k}{k}}k^2H_{2k}
 =\frac{2n(2n+1)\{2H_{2n}-H_n\}}{(2n-1)(2n-3)(2n-5)}\\
 &&\xqdn\:\:-\:\frac{n(96n^4-280n^3+60n^2+182n-13)}{(2n-1)^2(2n-3)^2(2n-5)^2}
 +\frac{2H_{2n}-H_n}{2}\sum_{k=0}^{n}(-4)^k\frac{\binm{n}{k}}{\binm{2k}{k}}k^2.
 \enm
Evaluating the series on the right hand side by
\eqref{chu-vandermonde-b}, we gain Theorem \ref{thm-c}.
\end{proof}

\section{The second family of summation formulas involving harmonic numbers}
\begin{thm}\label{thm-d}
 \bnm
 \sum_{k=0}^{n}(-4)^k\frac{\binm{n}{k}}{\binm{1+2k}{k}}H_{1+2k}
 =\frac{2H_{1+2n}-H_n}{2(4n^2-1)}-\frac{4n^3+8n^2+7n-2}{(4n^2-1)^2}.
 \enm
\end{thm}

\begin{proof}
Use \eqref{chu-vandermonde} and \eqref{chu-vandermonde-a} to achieve
  \bmn\label{chu-vandermonde-c}
\quad\sum_{k=0}^n(-1)^k(1+k)\binm{a+k}{k}\binm{b+n}{n-k}=\frac{b-a-1-an}{b-a-1+n}\binm{b-a-1+n}{n}.
 \emn
The case $a=\frac{x}{2}$, $b=\frac{1-x}{2}$ of
 it reads as
 \bnm
 \:\:\sum_{k=0}^n(-1)^k(1+k)\binm{\frac{x}{2}+k}{k}\binm{\frac{1-x}{2}+n}{n-k}=\frac{1+2x+nx}{1+2x-2n}\binm{-x-\frac{1}{2}+n}{n}.
 \enm
 Applying the derivative operator $\mathcal{D}_x$ to both sides of
the last equation, we have
 \bmn\label{harmornic-b}
 &&\xqdn\xxqdn\sum_{k=0}^{n}(-1)^k(1+k)\binm{\frac{1}{2}+n}{n-k}
  \bigg\{\frac{1}{2}H_k+\frac{1}{2}\sum_{i=1}^k\frac{1}{i+\frac{1}{2}}-\frac{1}{2}\sum_{i=1}^n\frac{1}{i+\frac{1}{2}}\bigg\}
   \nnm\\
&&\xqdn\xxqdn\:=\frac{1}{2n-1}\binm{-\frac{1}{2}+n}{n}\sum_{i=1}^n\frac{1}{i-\frac{1}{2}}-\frac{n(2n+3)}{(2n-1)^2}\binm{-\frac{1}{2}+n}{n}.
 \emn
In terms of the two relations
 \bmn
&&\xxqdn
(1+k)\binm{\frac{1}{2}+n}{n-k}=(1+n)4^{k-n}\binm{n}{k}\frac{\binm{1+2n}{n}}{\binm{1+2k}{k}},
 \label{relation-c}\\\label{relation-d}
&&\qqdn
\frac{1}{2}H_k+\frac{1}{2}\sum_{i=1}^k\frac{1}{i+\frac{1}{2}}=H_{1+2k}-1,
\emn
 \eqref{harmornic-b} can be manipulated as
 \bnm
\xqdn\qqdn\sum_{k=0}^{n}(-4)^k\frac{\binm{n}{k}}{\binm{1+2k}{k}}H_{1+2k}
 &&\xqdn\!=\frac{2-7n-8n^2-4n^3}{(4n^2-1)^2}+\frac{2H_{1+2n}-H_n}{4n^2-1}\\
&&\xqdn\!+\frac{2H_{1+2n}-H_n}{2}\sum_{k=0}^{n}(-4)^k\frac{\binm{n}{k}}{\binm{1+2k}{k}}.
 \enm
 Computing the series on the right hand side by \eqref{chu-vandermonde-c}, we attain Theorem \ref{thm-d}.
\end{proof}

\begin{thm}\label{thm-e}
 \bnm
 \xxqdn\:\sum_{k=0}^{n}(-4)^k\frac{\binm{n}{k}}{\binm{1+2k}{k}}kH_{1+2k}
 &&\xqdn=\:\frac{2n\{H_{n}-2H_{1+2n}\}}{(4n^2-1)(2n-3)}\\
 &&\xqdn+\:\:\frac{2n(8n^4+20n^3-2n^2-53n+8)}{(4n^2-1)^2(2n-3)^2}.
 \enm
\end{thm}

\begin{proof}
Utilize \eqref{chu-vandermonde-a} and \eqref{chu-vandermonde-b} to
obtain
 \bmn\label{chu-vandermonde-d}
&&\xxqdn\sum_{k=0}^n(-1)^kk(1+k)\binm{a+k}{k}\binm{b+n}{n-k}
 \nnm\\
&&\xxqdn\:=\:\frac{(a+1)(a+2-2b+an)n}{(b-a-2+n)(b-a-1+n)}\binm{b-a-1+n}{n}.
 \emn
The case $a=\frac{x}{2}$, $b=\frac{1-x}{2}$ of
 it is
 \bnm
 \quad\:\:\sum_{k=0}^n(-1)^kk(1+k)\binm{\frac{x}{2}+k}{k}\binm{\frac{1-x}{2}+n}{n-k}=
 \frac{(2+x)(2+3x+nx)n}{(1+2x-2n)(3+2x-2n)}\binm{-x-\frac{1}{2}+n}{n}.
 \enm
 Applying the derivative operator $\mathcal{D}_x$ to both sides of
the last equation, we get
 \bnm
 &&\xxqdn\sum_{k=0}^{n}(-1)^kk(1+k)\binm{\frac{1}{2}+n}{n-k}
  \bigg\{\frac{1}{2}H_k+\frac{1}{2}\sum_{i=1}^k\frac{1}{i+\frac{1}{2}}-\frac{1}{2}\sum_{i=1}^n\frac{1}{i+\frac{1}{2}}\bigg\}
   \nnm\\
&&\xxqdn\:=\frac{2n(4n^3+8n^2-13n-4)}{(2n-3)^2(2n-1)^2}\binm{-\frac{1}{2}+n}{n}-\frac{4n}{(2n-3)(2n-1)}\binm{-\frac{1}{2}+n}{n}\sum_{i=1}^n\frac{1}{i-\frac{1}{2}}.
 \enm
By means of \eqref{relation-c} and \eqref{relation-d}, the last
equation can be restated as
 \bnm
\sum_{k=0}^{n}(-4)^k\frac{\binm{n}{k}}{\binm{1+2k}{k}}kH_{1+2k}
 &&\xqdn\!=\frac{2n(8n^4+20n^3-2n^2-53n+8)}{(4n^2-1)^2(2n-3)^2}-\frac{4n\{2H_{1+2n}-H_n\}}{(4n^2-1)(2n-3)}\\
&&\xqdn\!+\frac{2H_{1+2n}-H_n}{2}\sum_{k=0}^{n}(-4)^k\frac{\binm{n}{k}}{\binm{1+2k}{k}}k.
 \enm
Calculating the series on the right hand side by
\eqref{chu-vandermonde-d}, we gain Theorem \ref{thm-e}.
\end{proof}

\begin{thm}\label{thm-f}
 \bnm
  &&\xxqdn\sum_{k=0}^{n}(-4)^k\frac{\binm{n}{k}}{\binm{1+2k}{k}}k^2H_{1+2k}
 =\frac{2n(4n-1)\{2H_{1+2n}-H_n\}}{(4n^2-1)(2n-3)(2n-5)}\\
 &&\xxqdn\:\:-\:\frac{2n(32n^6+160n^5-544n^4-560n^3+1482n^2-382n-17)}{(4n^2-1)^2(2n-3)^2(2n-5)^2}.
 \enm
\end{thm}

\begin{proof}
 It is not difficult to derive that
 \bnm
 &&\xxqdn\sum_{k=0}^n(-1)^kk^2(1+k)\binm{a+k}{k}\binm{b+n}{n-k}\\
 &&\xxqdn\:=\:(a+1)\sum_{k=1}^n(-1)^kk(1+k)\binm{a+k}{k-1}\binm{b+n}{n-k}\\
 &&\xxqdn\:=\:-(a+1)\sum_{k=0}^{n-1}(-1)^k(k+1)(k+2)\binm{a+1+k}{k}\binm{b+n}{n-1-k}\\
&&\xxqdn\:=\:-2(a+1)\sum_{k=0}^{n-1}(-1)^k\binm{a+1+k}{k}\binm{b+n}{n-1-k}\\
&&\xqdn\:-\:3(a+1)\sum_{k=0}^{n-1}(-1)^kk\binm{a+1+k}{k}\binm{b+n}{n-1-k}\\
&&\xqdn\:-\:(a+1)\sum_{k=0}^{n-1}(-1)^kk^2\binm{a+1+k}{k}\binm{b+n}{n-1-k}.\\
 \enm
Evaluate respectively the three series on the right hand side by
\eqref{chu-vandermonde}, \eqref{chu-vandermonde-a} and
\eqref{chu-vandermonde-b}, we achieve
  \bmn\label{chu-vandermonde-e}
&&\xxqdn\qdn\sum_{k=0}^n(-1)^kk^2(1+k)\binm{a+k}{k}\binm{b+n}{n-k}
=\binm{b-a-1+n}{n}\nnm\\
&&\xxqdn\qdn\:\times\:\frac{n(1+a)\{2b(1-b)-(1+a)(4+a-4b)n-a(1+a)n^2\}}{(b-a-3+n)(b-a-2+n)(b-a-1+n)}.
 \emn
The case $a=\frac{x}{2}$, $b=\frac{1-x}{2}$ of
 it can be written as
\bnm
 &&\xxqdn\sum_{k=0}^n(-1)^kk^2(1+k)\binm{\frac{x}{2}+k}{k}\binm{\frac{1-x}{2}+n}{n-k}=\binm{-x-\frac{1}{2}+n}{n}
 \\&&\xxqdn\:\times\:\frac{n(2+x)\{-2+8n+2n(7+n)x+(2+5n+n^2)x^2\}}{(1+2x-2n)(3+2x-2n)(5+2x-2n)}.
 \enm
 Applying the derivative operator $\mathcal{D}_x$ to both sides of
the last equation, we have \bnm
 &&\sum_{k=0}^{n}(-1)^kk^2(1+k)\binm{\frac{1}{2}+n}{n-k}
  \bigg\{\frac{1}{2}H_k+\frac{1}{2}\sum_{i=1}^k\frac{1}{i+\frac{1}{2}}-\frac{1}{2}\sum_{i=1}^n\frac{1}{i+\frac{1}{2}}\bigg\}
   \nnm\\
&&\:=\frac{4n(4n-1)}{(2n-1)(2n-3)(2n-5)}\binm{-\frac{1}{2}+n}{n}\sum_{i=1}^n\frac{1}{i-\frac{1}{2}}\\
&&\:-\:\frac{2n(16n^5+72n^4-372n^3+210n^2+196n-77)}{(2n-1)^2(2n-3)^2(2n-5)^2}\binm{-\frac{1}{2}+n}{n}.
 \enm
In accordance with \eqref{relation-c} and \eqref{relation-d}, the
last equation can be reformulated as
 \bnm
\sum_{k=0}^{n}(-4)^k\frac{\binm{n}{k}}{\binm{1+2k}{k}}k^2H_{1+2k}
 &&\xqdn\!=\frac{4n(4n-1)\{2H_{1+2n}-H_n\}}{(4n^2-1)(2n-3)(2n-5)}\\
  &&\xqdn\!-\frac{2n(32n^6+160n^5-544n^4-560n^3+1482n^2-382n-17)}{(4n^2-1)^2(2n-3)^2(2n-5)^2}\\
&&\xqdn\!+\frac{2H_{1+2n}-H_n}{2}\sum_{k=0}^{n}(-4)^k\frac{\binm{n}{k}}{\binm{1+2k}{k}}k^2.
 \enm
Computing the series on the right hand side by
\eqref{chu-vandermonde-e}, we attain Theorem \ref{thm-f}.
\end{proof}

\section{The third family of summation formulas involving harmonic numbers}
\begin{thm}\label{thm-g}
\bnm
 \quad\sum_{k=0}^{n}\bigg(-\frac{1}{4}\bigg)^k\binm{n}{k}\binm{2k}{k}H_{2k}
 =\frac{1}{2^{1+2n}}\binm{2n}{n}\{3H_n-4H_{2n}\}.
 \enm
\end{thm}

\begin{proof}
The case $a=\frac{x-1}{2}$, $b=-\frac{x}{2}$ of
 \eqref{chu-vandermonde} reads as
 \bnm
 \sum_{k=0}^{n}(-1)^k\binm{\frac{x-1}{2}+k}{k}\binm{-\frac{x}{2}+n}{n-k}
={\binm{-x-\tfrac{1}{2}+n}{n}}.
 \enm
Applying the derivative operator $\mathcal{D}_x$ to both sides of
it, we obtain
  \bmn\label{harmonic-c}
 \sum_{k=0}^{n}(-1)^k\binm{n}{k}\binm{-\tfrac{1}{2}+k}{k}
  \bigg\{\frac{1}{2}\sum_{i=1}^k\frac{1}{i-\frac{1}{2}}+\frac{1}{2}H_k-\frac{1}{2}H_n\bigg\}
  =-{\binm{-\tfrac{1}{2}+n}{n}}\sum_{i=1}^n\frac{1}{i-\frac{1}{2}}.
 \emn
According to \eqref{relation-b} and the relation:
 \bmn\label{relation-e}
\binm{-\frac{1}{2}+k}{k}=\frac{1}{4^k}\binm{2k}{k},
 \emn
 \eqref{harmonic-c} can be manipulated as
 \bnm
 \qquad\sum_{k=0}^{n}\bigg(-\frac{1}{4}\bigg)^k\binm{n}{k}\binm{2k}{k}H_{2k}= \frac{H_n}{2}\sum_{k=0}^{n}\bigg(-\frac{1}{4}\bigg)^k\binm{n}{k}\binm{2k}{k}
 -\frac{2H_{2n}-H_n}{4^n}\binm{2n}{n}.
 \enm
 Calculating the series on the right hand side by \eqref{chu-vandermonde}, we get Theorem \ref{thm-g}.
\end{proof}

\begin{thm}\label{thm-h}
\bnm
\qquad\sum_{k=0}^{n}\bigg(-\frac{1}{4}\bigg)^k\binm{n}{k}\binm{2k}{k}kH_{2k}
=\frac{n}{(1-2n)2^{1+2n}}\binm{2n}{n}\bigg\{3H_n-4H_{2n}-\frac{2+4n}{1-2n}\bigg\}.
 \enm
\end{thm}

\begin{proof}
The case $a=\frac{x-1}{2}$, $b=-\frac{x}{2}$ of
 \eqref{chu-vandermonde-a} is
 \bnm
\qquad
\sum_{k=0}^{n}(-1)^kk\binm{\frac{x-1}{2}+k}{k}\binm{-\frac{x}{2}+n}{n-k}
=\frac{(x+1)n}{2x+1-2n}\binm{-x-\tfrac{1}{2}+n}{n}.
 \enm
Applying the derivative operator $\mathcal{D}_x$ to both sides of
it, we have
  \bnm
 &&\xxqdn\xqdn\qqdn\sum_{k=0}^{n}(-1)^kk\binm{n}{k}\binm{-\tfrac{1}{2}+k}{k}
  \bigg\{\frac{1}{2}\sum_{i=1}^k\frac{1}{i-\frac{1}{2}}+\frac{1}{2}H_k-\frac{1}{2}H_n\bigg\}\\
  &&\xxqdn\xqdn\qqdn\:=\:\frac{n}{2n-1}\binm{-\tfrac{1}{2}+n}{n}\sum_{i=1}^n\frac{1}{i-\frac{1}{2}}
  -\frac{n(2n+1)}{(2n-1)^2}\binm{-\tfrac{1}{2}+n}{n}.
 \enm
In terms of \eqref{relation-b} and \eqref{relation-e},
 the last equation can be restated as
 \bnm
\xqdn\qdn\sum_{k=0}^{n}\bigg(-\frac{1}{4}\bigg)^k\binm{n}{k}\binm{2k}{k}kH_{2k}&&\xqdn\!=
\frac{n}{4^n(2n-1)}\binm{2n}{n}\bigg\{2H_{2n}-H_n-\frac{2n+1}{2n-1}\bigg\}\\
&&\xqdn\!+\:\frac{H_n}{2}\sum_{k=0}^{n}\bigg(-\frac{1}{4}\bigg)^k\binm{n}{k}\binm{2k}{k}k.
 \enm
 Evaluating the series on the right hand side by \eqref{chu-vandermonde-a}, we gain Theorem \ref{thm-h}.
\end{proof}

\begin{thm}\label{thm-i}
\bnm
 &&\xxqdn\sum_{k=0}^{n}\bigg(-\frac{1}{4}\bigg)^k\binm{n}{k}\binm{2k}{k}k^2H_{2k}
 =\frac{n^2}{4^n(2n-1)(2n-3)}\\
 &&\xxqdn\:\:\times\:\:\binm{2n}{n}\bigg\{\frac{3}{2}H_n-2H_{2n}+\frac{8n^3-4n^2-10n+3}{n(2n-1)(2n-3)}\bigg\}.
 \enm
\end{thm}

\begin{proof}
The case $a=\frac{x-1}{2}$, $b=-\frac{x}{2}$ of
 \eqref{chu-vandermonde-b} can be written as
 \bnm
\qquad\sum_{k=0}^{n}(-1)^kk^2\binm{\frac{x-1}{2}+k}{k}\binm{-\frac{x}{2}+n}{n-k}
=\frac{n(x+1)(n+x+nx)}{(2x+3-2n)(2x+1-2n)}\binm{-x-\tfrac{1}{2}+n}{n}.
 \enm
Applying the derivative operator $\mathcal{D}_x$ to both sides of
it, we achieve
  \bnm
 &&\xqdn\sum_{k=0}^{n}(-1)^kk^2\binm{n}{k}\binm{-\tfrac{1}{2}+k}{k}
  \bigg\{\frac{1}{2}\sum_{i=1}^k\frac{1}{i-\frac{1}{2}}+\frac{1}{2}H_k-\frac{1}{2}H_n\bigg\}\\
  &&\xqdn\:=\:\frac{n(8n^3-4n^2-10n+3)}{(2n-3)^2(2n-1)^2}\binm{-\tfrac{1}{2}+n}{n}
  -\frac{n^2}{(2n-3)(2n-1)}\binm{-\tfrac{1}{2}+n}{n}\sum_{i=1}^n\frac{1}{i-\frac{1}{2}}.
 \enm
By means of \eqref{relation-b} and \eqref{relation-e},
 the last equation can be reformulated as
 \bnm
&&\xxqdn\xxqdn\qqdn\sum_{k=0}^{n}\bigg(-\frac{1}{4}\bigg)^k\binm{n}{k}\binm{2k}{k}k^2H_{2k}=
\frac{H_n}{2}\sum_{k=0}^{n}\bigg(-\frac{1}{4}\bigg)^k\binm{n}{k}\binm{2k}{k}k^2\\
&&\xxqdn\xxqdn\qqdn\:-\:\frac{n^2}{4^n(2n-1)(2n-3)}\binm{2n}{n}\bigg\{2H_{2n}-H_n-\frac{8n^3-4n^2-10n+3}{n(2n-1)(2n-3)}\bigg\}.
 \enm
 Computing the series on the right hand side by \eqref{chu-vandermonde-b}, we attain  Theorem \ref{thm-i}.
\end{proof}

\section{The fourth family of summation formulas involving harmonic numbers}
\begin{thm}\label{thm-j}
\bnm
 &&\xqdn\sum_{k=0}^{n}\bigg(-\frac{1}{4}\bigg)^k\binm{n}{k}\binm{1+2k}{k}H_{1+2k}\\
 &&\xqdn\:=\:\frac{1}{2^{1+2n}(1+2n)}\binm{1+2n}{n}\bigg\{\frac{8+8n}{1+2n}+3H_n-4H_{1+2n}\bigg\}-\frac{1}{1+n}.
 \enm
\end{thm}

\begin{proof}
 It is ordinary to find that
 \bnm
 &&\qdn\xqdn\xxqdn\sum_{k=0}^n\frac{(-1)^k}{1+k}\binm{a+k}{k}\binm{b+n}{n-k}\\
 &&\qdn\xqdn\xxqdn\:=\:\sum_{k=1}^{n+1}\frac{(-1)^{k-1}}{k}\binm{a-1+k}{k-1}\binm{b+n}{1+n-k}\\
 &&\qdn\xqdn\xxqdn\:=\:-\frac{1}{a}\bigg\{\sum_{k=0}^{n+1}(-1)^k\binm{a-1+k}{k}\binm{b+n}{1+n-k}-\binm{b+n}{1+n}\bigg\}.
 \enm
Calculate the series on the right hand side by
\eqref{chu-vandermonde} to obtain
 \bmn\label{chu-vandermonde-f}
\sum_{k=0}^n\frac{(-1)^k}{1+k}\binm{a+k}{k}\binm{b+n}{n-k}=\frac{1}{a}\binm{b+n}{1+n}-\frac{1}{a}\binm{b-a+n}{1+n}.
 \emn
The case $a=\frac{x+1}{2}$, $b=-\frac{x}{2}$ of it reads as
 \bnm
 &&\xqdn\qqdn\sum_{k=0}^{n}\frac{(-1)^k}{1+k}\binm{\frac{x+1}{2}+k}{k}\binm{-\frac{x}{2}+n}{n-k}\\
&&\xqdn\qqdn\:=\:\frac{1+2x}{(1+n)(1+x)}{\binm{-x-\tfrac{1}{2}+n}{n}}-\frac{x}{(1+n)(1+x)}{\binm{-\frac{x}{2}+n}{n}}.
 \enm
Applying the derivative operator $\mathcal{D}_x$ to both sides of
the last equation, we get
  \bmn\label{harmonic-d}
 &&\xxqdn\xxqdn\sum_{k=0}^{n}\frac{(-1)^k}{1+k}\binm{n}{k}\binm{\tfrac{1}{2}+k}{k}
  \bigg\{\frac{1}{2}\sum_{i=1}^k\frac{1}{i+\frac{1}{2}}+\frac{1}{2}H_k-\frac{1}{2}H_n\bigg\}
   \nnm\\
  &&\xxqdn\xxqdn\:\:=\:\frac{1}{1+n}{\binm{-\tfrac{1}{2}+n}{n}}\bigg\{1-\sum_{i=1}^n\frac{1}{i-\frac{1}{2}}\bigg\}-\frac{1}{1+n}.
 \emn
In accordance with \eqref{relation-d} and the relation:
 \bmn\label{relation-f}
\frac{1}{1+k}\binm{\frac{1}{2}+k}{k}=\frac{1}{4^k}\binm{1+2k}{k},
 \emn
 \eqref{harmonic-d} can be manipulated as
 \bnm
 \sum_{k=0}^{n}\bigg(-\frac{1}{4}\bigg)^k\binm{n}{k}\binm{1+2k}{k}H_{1+2k}&&\xqdn\!
 =\frac{1}{4^n(1+2n)}\binm{1+2n}{n}\bigg\{\frac{3+2n}{1+2n}+H_n-2H_{1+2n}\bigg\}\\
&&\xqdn\! -\frac{1}{1+n}+\frac{2+H_n}{2}
\sum_{k=0}^{n}\bigg(-\frac{1}{4}\bigg)^k\binm{n}{k}\binm{1+2k}{k}.
 \enm
 Evaluating the series on the right hand side by \eqref{chu-vandermonde-f}, we gain Theorem \ref{thm-j}.
\end{proof}

\begin{thm}\label{thm-k}
\bnm
 &&\xqdn\sum_{k=0}^{n}\bigg(-\frac{1}{4}\bigg)^k\binm{n}{k}\binm{1+2k}{k}kH_{1+2k}=\frac{3n}{2^{1+2n}(1-4n^2)}\binm{1+2n}{n}\\
 &&\xqdn\:\:\times\:\bigg\{3H_n-
 4H_{1+2n}-\frac{16n}{1-4n^2}-\frac{2-14n}{3n}\bigg\}+\frac{1}{1+n}.
 \enm
\end{thm}

\begin{proof}
 It is routine to verify that
 \bnm
 &&\qqdn\xxqdn\sum_{k=0}^n(-1)^k\frac{k}{1+k}\binm{a+k}{k}\binm{b+n}{n-k}\\
 &&\qqdn\xxqdn\:=\:\sum_{k=0}^n(-1)^k\frac{(1+k)-1}{1+k}\binm{a+k}{k}\binm{b+n}{n-k}\\
 &&\qqdn\xxqdn\:=\:\sum_{k=0}^n(-1)^k\binm{a+k}{k}\binm{b+n}{n-k}\\
&&\qqdn\xxqdn\:-\:\,\sum_{k=0}^n(-1)^k\frac{1}{1+k}\binm{a+k}{k}\binm{b+n}{n-k}.
 \enm
Compute respectively the two series on the right hand side by
\eqref{chu-vandermonde} and \eqref{chu-vandermonde-f} to achieve
 \bmn\label{chu-vandermonde-g}
&&\xqdn\sum_{k=0}^n(-1)^k\frac{k}{1+k}\binm{a+k}{k}\binm{b+n}{n-k}
 \nnm\\
&&\xqdn\:=\:\frac{(1+a)n+b}{a(1+n)}\binm{b-a-1+n}{n}-\frac{1}{a}\binm{b+n}{1+n}.
 \emn
The case $a=\frac{x+1}{2}$, $b=-\frac{x}{2}$ of it is
 \bnm
 &&\xqdn\qqdn\sum_{k=0}^{n}(-1)^k\frac{k}{1+k}\binm{\frac{x+1}{2}+k}{k}\binm{-\frac{x}{2}+n}{n-k}\\
&&\xqdn\qqdn\:=\:\frac{(3+x)n-x}{(1+n)(1+x)}{\binm{-x-\tfrac{3}{2}+n}{n}}+\frac{x}{(1+n)(1+x)}{\binm{-\frac{x}{2}+n}{n}}.
 \enm
Applying the derivative operator $\mathcal{D}_x$ to both sides of
the last equation, we have
  \bnm
 &&\xxqdn\sum_{k=0}^{n}(-1)^k\frac{k}{1+k}\binm{n}{k}\binm{\tfrac{1}{2}+k}{k}
  \bigg\{\frac{1}{2}\sum_{i=1}^k\frac{1}{i+\frac{1}{2}}+\frac{1}{2}H_k-\frac{1}{2}H_n\bigg\}
   \nnm\\
  &&\xxqdn\:\:=\:\frac{1}{1+n}-\frac{3n}{1+n}{\binm{-\tfrac{3}{2}+n}{n}}\bigg\{\sum_{i=1}^n\frac{1}{i-\frac{3}{2}}+\frac{1+2n}{3n}\bigg\}.
 \enm
According to \eqref{relation-d} and \eqref{relation-f}, the last
equation can be restated as
 \bnm
&&\xqdn\sum_{k=0}^{n}\bigg(-\frac{1}{4}\bigg)^k\binm{n}{k}\binm{1+2k}{k}kH_{1+2k}\\
 &&\xqdn\:=\:\frac{1}{1+n}-\frac{3n}{4^n(1-4n^2)}\binm{1+2n}{n}\bigg\{2H_{1+2n}-H_n+\frac{1-4n+20n^2+16n^3}{3n(1-4n^2)}\bigg\}\\
&&\xqdn\:+\:\,\frac{2+H_n}{2}
\sum_{k=0}^{n}\bigg(-\frac{1}{4}\bigg)^k\binm{n}{k}\binm{1+2k}{k}k.
 \enm
 Calculating the series on the right hand side by \eqref{chu-vandermonde-g}, we attain Theorem \ref{thm-k}.
\end{proof}

\begin{thm}\label{thm-l}
\bnm
 &&\xqdn\qdn\sum_{k=0}^{n}\bigg(-\frac{1}{4}\bigg)^k\binm{n}{k}\binm{1+2k}{k}k^2H_{1+2k}=\frac{3n(2-3n)}{(3-2n)(1-4n^2)4^n}\binm{1+2n}{n}\\
 &&\xqdn\qdn\:\times\:\bigg\{2H_{1+2n}-\frac{3}{2}H_n-\frac{3-8n-12n^2}{1-4n^2}+\frac{9+n(3-25n+6n^2)}{3n(3-2n)(2-3n)}\bigg\}
 -\frac{1}{1+n}.
 \enm
\end{thm}

\begin{proof}
 It is not difficult to see that
 \bnm
 &&\qqdn\xxqdn\sum_{k=0}^n(-1)^k\frac{k^2}{1+k}\binm{a+k}{k}\binm{b+n}{n-k}\\
 &&\qqdn\xxqdn\:=\:\sum_{k=0}^n(-1)^k\frac{(k^2-1)+1}{1+k}\binm{a+k}{k}\binm{b+n}{n-k}\\
 &&\qqdn\xxqdn\:=\:\sum_{k=0}^n(-1)^kk\binm{a+k}{k}\binm{b+n}{n-k}\\
 &&\qqdn\xxqdn\:-\:\sum_{k=0}^n(-1)^k\binm{a+k}{k}\binm{b+n}{n-k}\\
&&\qqdn\xxqdn\:+\:\,\sum_{k=0}^n(-1)^k\frac{1}{1+k}\binm{a+k}{k}\binm{b+n}{n-k}.
 \enm
Evaluate respectively the three series on the right hand side by
\eqref{chu-vandermonde-a}, \eqref{chu-vandermonde} and
\eqref{chu-vandermonde-f} to obtain
 \bmn\label{chu-vandermonde-i}
&&\qqdn\sum_{k=0}^n(-1)^k\frac{k^2}{1+k}\binm{a+k}{k}\binm{b+n}{n-k}=\frac{1}{a}\binm{b+n}{1+n}
 \nnm\\
&&\qqdn\:+\:\frac{(1+a)(an+a-1)n-(1+a-2n-an)b+b^2}{a(1+n)(1+a-b-n)}\binm{b-a-1+n}{n}.
 \emn
The case $a=\frac{x+1}{2}$, $b=-\frac{x}{2}$ of it can be written as
\bnm
 &&\xxqdn\qqdn\sum_{k=0}^{n}(-1)^k\frac{k^2}{1+k}\binm{\frac{x+1}{2}+k}{k}\binm{-\frac{x}{2}+n}{n-k}\\
&&\xxqdn\qqdn\:=\frac{(3+x)^2n^2-(1+x)(6+x)n+x(3+2x)}{(1+n)(1+x)(3+2x-2n)}{\binm{-x-\tfrac{3}{2}+n}{n}}\\
&&\xxqdn\qqdn\:-\:\frac{x}{(1+n)(1+x)}{\binm{-\frac{x}{2}+n}{n}}.
 \enm
Applying the derivative operator $\mathcal{D}_x$ to both sides of
the last equation, we get
  \bnm
 &&\sum_{k=0}^{n}(-1)^k\frac{k^2}{1+k}\binm{n}{k}\binm{\tfrac{1}{2}+k}{k}
  \bigg\{\frac{1}{2}\sum_{i=1}^k\frac{1}{i+\frac{1}{2}}+\frac{1}{2}H_k-\frac{1}{2}H_n\bigg\}
   \nnm\\
  &&\:=\:\frac{3n(2-3n)}{(1+n)(3-2n)}{\binm{-\tfrac{3}{2}+n}{n}}
  \bigg\{\sum_{i=1}^n\frac{1}{i-\frac{3}{2}}+\frac{9+n(3-25n+6n^2)}{3n(3-2n)(2-3n)}\bigg\}-\frac{1}{1+n}.
 \enm
In terms of \eqref{relation-d} and \eqref{relation-f}, the last
equation can be reformulated as
 \bnm
&&\xxqdn\sum_{k=0}^{n}\bigg(-\frac{1}{4}\bigg)^k\binm{n}{k}\binm{1+2k}{k}k^2H_{1+2k}=\frac{3n(2-3n)}{(3-2n)(1-4n^2)4^n}\binm{1+2n}{n}\\
&&\xxqdn\:\times\:\bigg\{2H_{1+2n}-H_n-\frac{2-8n-8n^2}{1-4n^2}+\frac{9+n(3-25n+6n^2)}{3n(3-2n)(2-3n)}\bigg\}-\frac{1}{1+n}\\
&&\xxqdn\:+\:
\frac{2+H_n}{2}\sum_{k=0}^{n}\bigg(-\frac{1}{4}\bigg)^k\binm{n}{k}\binm{1+2k}{k}k^2.
 \enm
 Reckoning the series on the right hand side by \eqref{chu-vandermonde-i}, we gain Theorem \ref{thm-l}.
\end{proof}

 \textbf{Acknowledgments}

The work is supported by the National Natural Science Foundation of
China (No. 11661032).



\end{document}